\documentclass[11pt]{amsart}
\usepackage{amsthm,amsmath,amssymb}
\usepackage[colorlinks = true, citecolor = blue]{hyperref}
\usepackage{framed}

\title[Groups of Morley rank $5$]{Simple groups of Morley rank $5$ are bad}
\author{Adrien Deloro}
\address{Sorbonne Universit\'es, UPMC, Institut de Math\'ematiques de Jussieu -- Paris rive gauche, Case 247, 4 place Jussieu, 75252 Paris, France}
\email{adrien.deloro@imj-prg.fr}
\author{Joshua Wiscons}
\address{Department of Mathematics and Statistics\\
California State University, Sacramento\\
Sacramento, CA 95819, USA}
\email{joshua.wiscons@csus.edu}
\thanks{The second author was partially supported by the National Science Foundation under grant No. OISE-1064446.}

\newtheorem*{theorem*}{Theorem}
\newtheorem*{corollary*}{Corollary}
\newtheorem*{algconj}{Algebraicity Conjecture}
\newtheorem*{fact*}{Fact}
\newtheorem{theorem}{Theorem}[section]
\newtheorem{corollary}[theorem]{Corollary}
\newtheorem{lemma}[theorem]{Lemma}
\newtheorem{proposition}[theorem]{Proposition}
\newtheorem{fact}[theorem]{Fact}
\newtheorem{observation}[theorem]{Observation}

\theoremstyle{definition}
\newtheorem{definition}[theorem]{Definition}
\newtheorem*{setup}{Setup}
\newtheorem{remark}[theorem]{Remark}

\newcommand{\cibo}{\textbf{CiBo}}
\newcommand{\textdef}[1]{\textit{#1}}

\DeclareMathOperator{\pssl}{(P)SL} 
\DeclareMathOperator{\psl}{PSL}
\DeclareMathOperator{\ssl}{SL}
\DeclareMathOperator{\pgl}{PGL} 
\DeclareMathOperator{\agl}{AGL} 
 
\DeclareMathOperator{\so}{SO} 
\DeclareMathOperator{\rk}{rk} 
\DeclareMathOperator{\pr}{pr} 
\DeclareMathOperator{\charac}{char}

\begin{document}
\begin{abstract}
By exploiting the geometry of involutions in $N_\circ^\circ$-groups of finite Morley rank, we show that any simple group of Morley rank $5$ is a bad group all of whose proper definable connected subgroups are nilpotent of rank at most $2$. The main result is then used to catalog the nonsoluble connected groups of Morley rank $5$. 
\end{abstract}

\maketitle

The groups of finite Morley rank form a model-theoretically natural and important class of groups that are equipped with a notion of dimension generalizing the usual Zariski dimension for affine algebraic groups. This class of groups is known to contain many nonalgebraic examples, but nevertheless, there is an extremely tight connection with algebraic geometry witnessed by, among other things, the following conjecture of Cherlin and Zilber.

\begin{algconj}
An infinite simple group of finite Morley rank is isomorphic to an affine algebraic group over an algebraically closed field.
\end{algconj}

This conjecture appears in Cherlin's early paper \cite{ChG79} analyzing groups of Morley rank at most $3$, but even in that very small rank setting, a potential  nonalgebraic simple group resisted all efforts to kill it until the very recent work of Fr\'{e}con \cite{FrO16} (see also \cite{Wa17}). The configuration Cherlin encountered is that of a so-called \textdef{bad group}, which is defined to be a nonsoluble group of finite Morley rank all of whose proper definable soluble subgroups are nilpotent-by-finite. Despite the spectre of bad groups, a fair amount of progress has been made on the Algebraicity Conjecture. Most notable is the deep result of Alt\i{}nel, Borovik, and Cherlin from 2008 that established the conjecture for those groups containing an infinite elementary abelian $2$-group, \cite{ABC08}. The  remaining cases to investigate turn out to be:
\begin{description}
\item[Odd type] when the group contains a copy of the Pr\"ufer $2$-group $\mathbb{Z}_{2^\infty}$ but no infinite elementary abelian $2$-group, and
\item[Degenerate type] when the group contains no involutions at all.
\end{description}

It is a theorem, involving ideas by Borovik, Corredor, Nesin, and Poizat, that bad groups are of degenerate type. As bad groups of rank $3$ have now been dealt with by Fr\'econ, some understanding of degenerate type groups is no longer such an unlikely dream. One may also hope to resolve the Algebraicity Conjecture for odd type groups, and here there exists a reasonably solid theory. But even in the presence of involutions, some extremely tight, nonalgebraic configurations have arisen. They first appeared in Cherlin and Jaligot's investigation of tame minimal simple odd type groups \cite{ChJa04} and persisted into the more recent and more general $N_\circ^\circ$-context studied by Jaligot and the first author in \cite{DeJa16}. These configurations have been named \cibo$_1$, \cibo$_2$, and \cibo$_3$ and will be described in \S\ref{s:NooCiBo} below. 

With the hope of shedding new light on the \cibo\ configurations, and hence on the Algebraicity Conjecture for odd type groups, we decided to study these pathologies in a small rank setting. Groups of rank $4$ had already been addressed by the second author in \cite{WiJ14a}, so we took up rank $5$. Unlike in \cite{WiJ14a}, we did not shy away from the heavy machinery. We were indeed successful in killing the \cibo\ configurations in rank $5$, which is perhaps not surprising, but more importantly, our analysis hints at some general techniques. The takeaway message seems to be, as has been observed several times before, that the geometry of involutions is a powerful tool for studying tight configurations. Our main result, which may in the future be pushed further, is as follows. 

\begin{theorem*}
A simple group of Morley rank $5$ is a bad group all of whose proper definable connected subgroups are nilpotent of rank at most $2$.
\end{theorem*}

Combining this theorem with several existing results, we obtain, with little effort, a classification of groups of rank $5$ (up to bad groups of low rank). Regarding notation, $F^\circ(G)$ denotes the connected Fitting subgroup of $G$ (see Section~\ref{s:uniqueness}), $G'$ denotes the commutator subgroup of $G$, and we write $G=A*B$ if $A$ and $B$ commute and generate $G$. Also, recall that a group is called \textdef{quasisimple} if it is perfect, and modulo its centre, it is simple. 

\begin{corollary*}
If $G$ is a nonsoluble connected group of Morley rank $5$, then $F:=F^\circ(G)$ has rank at most $2$, and $G$ is classified as follows.
\begin{itemize}
\item If $\rk F = 2$, then either
\begin{itemize}
\item $G/F \cong \ssl_2(K)$ acting naturally on $F\cong K^2$ with the extension split whenever $\charac K \neq 2$, or
\item $G = F * G''$ with $G''\cong\pssl_2(K)$.
\end{itemize}
\item If $\rk F = 1$, then either
\begin{itemize}
\item $G = R * G''$ with $R/Z(R) \cong \agl_1(L)$ and $G''\cong\pssl_2(K)$,
\item $G = F * G'$ with $G'$ quasisimple and bad of rank $4$, or
\item $G$ is quasisimple and bad of rank $5$.
\end{itemize}
\item If $\rk F = 0$, then $G'$ is quasisimple and bad group of rank $4$ or $5$.
\end{itemize}
Consequently, $\rk F \ge 1$ whenever $G$ contains an involution.
\end{corollary*}

We have thus far presented our work as simply a testing ground for the study of \cibo$_1$, \cibo$_2$, and \cibo$_3$, but it may also have applications to the current study of limits to the so-called degree of generic transitivity of a permutation group of finite Morley rank. This was indeed the case with the classification of groups of Morley rank $4$, see \cite{AlWi15} and \cite{BoDe16}.

\section{Background} 
Here we simply collect a handful of results on ``small'' groups of finite Morley rank. For a general reference on groups of finite Morley rank, we recommend \cite{PoB87}, \cite{BoNe94}, or \cite{ABC08}. We assume familiarity with basic definability and connectedness notions \cite[\S5.2]{BoNe94}, Zilber's Field Theorem \cite[Theorem 9.1]{BoNe94}, and some standard involution-related techniques \cite[\S10]{BoNe94} such as $2$-Sylow theory and Brauer-Fowler estimates.

For a group $G$, let $I(G)$ stand for the set of its involutions. The Pr\"ufer $2$-rank $\pr_2(G)$ is the maximal $r$ such that $\bigoplus_r \mathbb{Z}_{2^\infty}$ is contained in a Sylow $2$-subgroup of $G$; $\mathbb{Z}_{2^\infty}$ being the Pr\"ufer $2$-group. The $2$-rank $m_2(G)$ is simply the maximal rank (in the algebraic sense) of an elementary abelian $2$-subgroup of $G$.

\subsection{\texorpdfstring{$N_\circ^\circ$}{N}-groups and \cibo}\label{s:NooCiBo}

\begin{definition}
A group of finite Morley rank is an \textdef{$N_\circ^\circ$-group} if $N^\circ(A)$ remains soluble whenever $A$ is connected and soluble.
\end{definition}

The condition for being an $N_\circ^\circ$-group is a local analog of minimal simplicity and implies a sense of smallness of the group. The only nonsoluble, \textit{algebraic} $N_\circ^\circ$-groups are those of the form $\pssl_2$. The \cibo\ configurations are, at present, unavoidable complications in the study of $N_\circ^\circ$-groups for which the \textbf{C}entralizer of an \textbf{i}nvolution is a \textbf{Bo}rel subgroup. (A \textdef{Borel subgroup} is a maximal connected definable soluble subgroup.)

\begin{fact}[Special case of \protect{\cite{DeJa16}}]\label{f:DJ}
Let $G$ be an infinite, nonsoluble $N_\circ^\circ$-group of finite Morley rank with involutions. Further assume that $C^\circ(i)$ is soluble for all $i\in I(G)$. Then $G$ is of one of the following types:
\begin{description}
\item[CiBo$_1$] $\pr_2(G) = m_2(G) = 1$, and $C(i)$ is a self-normalizing Borel subgroup of $G$;
\item[CiBo$_2$] $\pr_2(G) = 1$, $m_2(G) = 2$, $C^\circ(i)$ is an abelian Borel subgroup of $G$ inverted by any involution in $C(i)-\{i\}$, and $\rk G = 3\cdot\rk C(i)$;
\item[CiBo$_3$] $\pr_2(G) = m_2(G) = 2$,  $C(i)$ is a self-normalizing Borel subgroup of $G$, and if $i\neq j\in I(G)$, then $C(i) \neq C(j)$;
\item[Algebraic] $G \cong \psl_2(K)$.
\end{description}
\end{fact}

\subsection{Small groups and small actions}\label{s:small}
Here we briefly summarize the existing results about groups of small Morley rank. 

\begin{fact}[\cite{ReJ75}]\label{f.rankOneGroups}
If $A$ is a connected group of rank $1$, then $A$ is either a divisible abelian group or an elementary abelian $p$-group for some prime $p$.
\end{fact}

\begin{fact}[\cite{ChG79}]\label{f.rankTwoGroups}
If $B$ is a connected group of rank $2$, then $B$ is soluble. If $B$ is nilpotent and nonabelian, then $B$ has exponent $p$ or $p^2$ for some prime $p$, and if $B$ is nonnilpotent, then $B/Z(B) \cong K_+ \rtimes K^\times$ for some algebraically closed field $K$.
\end{fact}

A nilpotent $p$-group of bounded exponent is called \emph{$p$-unipotent}; such subgroups play an important role in the analysis of soluble groups and their intersections.

The main result for groups of rank $3$ takes the following very satisfying form as a result of the aforementioned efforts of Fr\'econ. Sadly, the situation in rank $4$ is not yet so clear. 

\begin{fact}[{\cite{ChG79}} and {\cite{FrO16}}]\label{f.rankThreeGroups}
A simple group of rank $3$ is isomorphic to $\psl_2(K)$ for some algebraically closed field $K$.
\end{fact}

\begin{fact}[{\cite[Theorem~A]{WiJ14a}}]\label{f.rankFourGroups}
A simple group of rank $4$  is a bad group whose definable proper subgroups have rank at most $1$.
\end{fact}

Finally, the ubiquity of the following result of Hrushovski in the study of groups of small rank is hard to overstate; despite non-trivial intersection with the above, we prefer to state it separately.

\begin{fact}[Hrushovski, see \protect{\cite[Theorem~11.98]{BoNe94}}]\label{fact.Hru}
Suppose that $H$ is a group of finite Morley rank acting faithfully, transitively and definably on a definable  set $X$ of rank and degree $1$. Then $\rk H \le 3$, and if $\rk H >1$, there is a definable algebraically closed field $K$ such that the action of $H$ on $X$ is equivalent to $\agl_1(K) = K_+ \rtimes K^\times$ acting on $\mathbb{A}_1(K) = K$  or $\pgl_2(K)$ acting on $\mathbb{P}^1(K) = K\cup\{\infty\}$.
\end{fact}

\section{Reaching the \texorpdfstring{$N_\circ^\circ$}{Noo}-analysis} 


\begin{setup}
Let $G$ be a simple group of Morley rank $5$.
\end{setup}

The goal of the present section is to show that the Deloro-Jaligot analysis of Fact~\ref{f:DJ} applies to $G$; that is, we aim to establish the following proposition.

\begin{proposition}\label{p:GisCibo}
The group $G$ is an $N_\circ^\circ$-group that either has no involutions or is of type \cibo$_1$,  \cibo$_2$, or  \cibo$_3$.
\end{proposition}

First notice that by Fact~\ref{fact.Hru}, $G$ has no definable subgroups of rank $4$.
Consequently, if $A \neq 1$ is a connnected soluble subgroup of $G$, then $N^\circ(A)/A$ has rank at most $2$, so $N^\circ(A)$ is soluble by Fact~\ref{f.rankTwoGroups}. Combining this with the classification of even and mixed type simple groups from \cite{ABC08}, we easily arrive at the following lemma.

\begin{lemma}\label{l:GisOddOrDeg}
The group $G$  is an $N_\circ^\circ$-group of odd or degenerate type.
\end{lemma}

To apply the Deloro-Jaligot analysis to $G$, it remains to prove solubility of centralisers of involutions---this, however, this requires knowledge of intersections of Borel subgroups, which is where we begin. 

\subsection{Uniqueness principles}\label{s:uniqueness}
Since $G$ is an $N_\circ^\circ$-group, $G$ enjoys certain general so-called ``uniqueness principles'' \cite[Fact~8]{DeJa16}. However, the small rank context allows for a much stronger version with the effect of completely short-cutting Burdges' elaborate unipotence theory, on which we shall therefore not dwell.

Recall that the connected Fitting subgroup $F^\circ(H)$ of a group $H$ is its characteristic subgroup generated by all definable, connected, normal nilpotent subgroups of $H$; when $H$ has finite Morley rank, $F^\circ(H)$ is definable and nilpotent \cite[Theorem~7.3]{BoNe94}. Also, for soluble connected $H$ of finite Morley rank, one has $H' \leq F^\circ(H)$ \cite[Corollary~9.9]{BoNe94}, and $C_H^\circ(F^\circ(H)) \leq F^\circ(H)$ \cite[Proposition~7.4]{BoNe94}; finally, $H/F^\circ(H)$ is a divisible abelian group \cite[Theorem~9.21]{BoNe94}.
Borel subgroups were defined at the beginning of \S\ref{s:NooCiBo}.

\begin{lemma}\label{l:uniqueness}
If $B_1 \neq B_2$ are two Borel subgroups of $G$, then $(F^\circ(B_1) \cap F^\circ(B_2))^\circ = 1$. 
\end{lemma}
\begin{proof}
Let $X = (F^\circ(B_1) \cap F^\circ(B_2))^\circ$, and suppose that $X \neq 1$.

First notice that $X$ cannot be normal in both $B_1$ and $B_2$. Otherwise, since $G$ is an $N_\circ^\circ$-group and by definition of a Borel subgroup, one has $B_1 = N^\circ(X) = B_2$, a contradiction.

As a matter of fact $X$ cannot be normal in either. If say $X \trianglelefteq B_1$, then $X$ cannot be normal in $B_2$, so $X < F^\circ(B_2)$. By the normaliser condition \cite[Lemma 6.3]{BoNe94}, one has $K_2 := N_{F^\circ(B_2)}(X) > X$. Of course $K_2 \leq N^\circ(X) = B_1$. If $K_2 \trianglelefteq B_1$ then $K_2 \leq F^\circ(B_1)$, against the definition of $X$. In particular $B_1$ has rank $3$ and is non-nilpotent, $K_2$ has rank $2$, and $X$ has rank $1$. Now $X$ is easily (or classically: \cite[\S6.1, Exercise~5]{BoNe94}) seen to be central in $K_2$. Moreover, if $F^\circ(B_1) = X$, then $B_1' \leq F^\circ(B_1) = X \leq K_2$ so $K_2 \trianglelefteq B_1$, a contradiction as we know. This shows that $F^\circ(B_1)$ has rank $2$ as well, and therefore $X$ is central in $\langle K_2, F^\circ(B_1)\rangle = B_1$.
Since $B_1$ is not nilpotent, the rank $2$ factor group $\beta = B_1/X$ is not either, so by Fact~\ref{f.rankTwoGroups}, its structure is known modulo a finite centre: isomorphic to some $K_+\rtimes K^\times$. Since $K_2 \not \leq F^\circ(B_1)$, $K_2$ covers $\beta/F(\beta)$, so $K_2$ contains an infinite divisible torsion subgroup. Thus, by a standard ``torsion-lifting'' argument \cite[\S5.5, Exercise~11]{BoNe94}, there is at most one definable infinite subgroup of $K_2$ containing no divisible torsion, so if it exists, it must be definably characteristic in $K_2$.

Now, still assuming $X \trianglelefteq B_1$, recall that $X$ cannot be normal in $B_2$. However, whether $B_2$ is nilpotent or not, one has $K_2 \trianglelefteq B_2$, so $X$ cannot be definably characteristic in $K_2$. Thus, as observed above, $X$ must contain divisible torsion. Since $\rk X = 1$, $X$ is a decent torus so is contained in the (unique) maximal decent torus of $F^\circ(B_2)$. By the so-called rigidity of tori \cite[Theorem~6.16]{BoNe94} (or \cite[\S9.2, Exercise~3]{BoNe94}), $X$ is central in $B_2$. This is a contradiction, and we conclude that $X$ is normal in neither $B_1$ nor $B_2$.

Consider $N := N^\circ(X)$. As above $K_2 > X$ and $K_1 := N^\circ_{F^\circ(B_1)}(X) > X$, so $N = \langle K_1, K_2\rangle$ is a rank $3$ soluble group. It is a Borel subgroup in which $X = (F^\circ(B_1) \cap F^\circ(N))^\circ$ is normal, a final contradiction.
\end{proof}

\begin{corollary}\label{c:uniqueness}
If $B$ is a rank $3$ Borel subgroup of $G$, then $\rk F^\circ(B) = 2$.
\end{corollary}
\begin{proof}
Rank considerations ensure that $B$ is not nilpotent, as otherwise $B = F^\circ(B)$ and any conjugate would have an infinite intersection, contradicting the previous lemma. Hence, $\rk F^\circ(B) \le 2$. Additionally, by Zilber's Field Theorem and the aforementioned fact that $C_H^\circ(F^\circ(H)) \leq F^\circ(H)$ for connected soluble $H$ of finite Morley rank, one has $\rk F^\circ(B)\neq 1$.
\end{proof}

\begin{corollary}\label{c:BorelIntersections}
If $B_1 \neq B_2$ are two rank $3$ Borel subgroups of $G$, then either $\rk(B_1 \cap B_2) = 1$, or there exist commuting involutions $i$ and $j$ for which $B_1 = C^\circ(i)$ and $B_2 = C^\circ_2(j)$.
\end{corollary}
\begin{proof}
As $G$ has rank $5$, $\rk(B_1 \cap B_2) \ge 1$. Now, assume that $H:= (B_1 \cap B_2)^\circ$ has rank $2$. 
Consider the canonical map from $H$ to $B_1/F^\circ(B_1) \times B_2/F^\circ (B_2)$; by Lemma~\ref{l:uniqueness} it has a finite kernel. Thus, the connected group $H'$ is finite, so $H$ is abelian. Also, $B_1 = H\cdot F^\circ(B_1)$ since otherwise $H = F^\circ(B_1)$ implying that the rank $3$ group $B_2$ contains the rank $2$ subgroups $F^\circ(B_1)$ and $F^\circ(B_2)$, which by Lemma~\ref{l:uniqueness} have a finite intersection.

Set $A_1 := (H\cap F^\circ(B_1))^\circ$ and $A_2 := (H\cap F^\circ(B_2))^\circ$. By Lemma~\ref{l:uniqueness}, $(A_1 \cap A_2)^\circ = 1$. Since $F^\circ(B_1)$ has rank $2$,  $A_1$ must be central in $F^\circ(B_1)$, and as $H$ is abelian, we find that $A_1$ is central in $B_1 = H\cdot F^\circ(B_1)$. Moreover, $B_1/A_1$ must be of the form $K_+\rtimes K^\times$ modulo a finite centre, so $H/A_1 \cong K^\times$ for some algebraically closed field $K$ of characteristic not $2$ (using  Lemma~\ref{l:GisOddOrDeg}).
But, $H = A_1 \cdot A_2$, so $A_2$ contains some involution $j$, and $B_2 = C^\circ(j)$. An analogous argument shows that $B_1 = C^\circ(i)$ for some involution $i\in A_1$, which commutes with $j$, so we are done.
\end{proof}

\subsection{Solubility of \texorpdfstring{$C^\circ(i)$}{C(i)}}\label{s:solubilityC}

\begin{lemma}\label{l:C(i)}
For $i \in I(G)$, $C^\circ(i)$ is soluble.
\end{lemma}
\begin{proof}
Suppose not. Then $C_i := C^\circ(i)$ has rank $3$, and it must be that $C_i \simeq \mathrm{SL}_2(\mathbb{K})$ in characteristic not $2$. By torality \cite[Theorem 3]{BuCh08}, the Sylow $2$-subgroup of $G$ is like that of $\mathrm{SL}_2(\mathbb{K})$ and involutions are conjugate.

Let $j \in I(G)$ be a generic conjugate of $i$ and $A = (C_i\cap C_j)^\circ$. Since $i$ is the only involution in $C_i$, the group $A$ contains no involutions: it is therefore an algebraic unipotent subgroup of $C_i$, contained in the Borel subgroup $B_i = N_{C_i}(A)$ of $C_i$ (resp. $B_j = N_{C_j}(A)$ in $C_j$).

Since $i \neq j$, one has $B_i \neq B_j$; hence $N := N^\circ(A)$ is a Borel subgroup of $G$ of rank $3$. If the Fitting subgroup $F := F^\circ(N)$ contains some involution $k$ then since we are in odd type and by the rigidity of tori, $k$ is central in $N$, so using the structure of the Sylow $2$-subgroup, $i = k = j$ which is a contradiction. Moreover, $F$ has rank $2$ as we know, and $A = B_i' \leq N' \leq F^\circ(N)$.

Now let $k$ be an involution generic over $i$ and $j$. If $A_k := (C_k \cap N)^\circ \neq 1$ contains an involution, it can only be the involution $k$, which then normalises $A$: against genericity. Hence again $A_k$ is a unipotent subgroup of $C_k$. Lifting torsion, $A_k F$ has no involutions, so it is proper in $N$. Therefore $A_k \leq F$.

By conjugacy of involutions in $G$ and of unipotent subgroups in $C_i$, the groups $A$ and $A_k$ are however conjugate, so we know the structure of $N_k = N^\circ(A_k)$: a Borel subgroup with $A_k \leq F^\circ(N_k)$. Hence $A_k \leq F\cap F^\circ(N_k)$, and Lemma \ref{l:uniqueness} forces $N = N_k$, so $k$ normalises $N$, against genericity again.
\end{proof}

\begin{proof}[Proof of Proposition~\ref{p:GisCibo}]
Fact~\ref{f:DJ} now applies to $G$, and as rank considerations rule out the possibility that $G$ is of the form $\psl_2$, Proposition~\ref{p:GisCibo} is proven. Incidentally, rank considerations also immediately rule out \cibo$_2$. We shall return to this in \S\ref{s:CiBo2}.
\end{proof}

\section{Using the \texorpdfstring{$N_\circ^\circ$}{Noo}-analysis} 

%

By Proposition~\ref{p:GisCibo}, the simple rank $5$ group $G$ has no involutions or is of type \cibo$_k$ for $k \in \{ 1, 2, 3\}$. Different cases beg for distinct methods.

\setcounter{subsection}{-1}
\subsection{The bad case}

Let us quickly address the case without involutions.

\begin{proposition}\label{p.bad}
If $G$ has no involutions, then $G$ is a bad group all of whose proper definable connected subgroups are nilpotent of rank at most $2$.
\end{proposition}
\begin{proof}
Suppose $G$ has a subgroup $H$ of rank $3$ and divide into cases. 

First, assume that $H$ is non-soluble. Then having no involutions, $H$ is a bad group of rank $3$: hence for $g \notin N(H)$, the intersection $H \cap H^g$ has rank $1$. So, the orbit of $H^g$ under conjugation by $H$ has rank $2$, and as $X := \{H^g: g \in G\}$ has rank $2$ and degree $1$, we find that the action of $G$ on $X$ is generically $2$-transitive, i.e. that there is a unique generic orbit on $X\times X$. Lifting torsion, this creates an involution in $G$, a contradiction. (Of course there is Fr\'econ's Theorem as well.)

Now suppose that $H$ is soluble. To avoid the same contradiction as above, for generic $g$, the intersection $H \cap H^g$ has rank $2$, but this contradicts Corollary~\ref{c:BorelIntersections}.

So $G$ has no definable, connected, proper subgroup of rank $3$. Having no involutions, its rank $2$ subgroups are nilpotent; $G$ is a bad group.
\end{proof}

\subsection{Killing \texorpdfstring{\cibo$_1$}{CiBo1}: Moufang sets}

We target the following proposition. 

\begin{proposition}\label{p:cibo1}
The group $G$ cannot have type \cibo$_1$.
\end{proposition}

\begin{setup}
Assume $G$ has type \cibo$_1$.
\end{setup}

The two main ingredients of our analysis in this case are (a small fragment of) the theory of Moufang sets and the Brauer-Fowler computation. The latter will be used later; for the moment we focus on the former.
Moufang sets encode the essence of so-called split $BN$-pairs of (Tits) rank $1$, and as such, they sit at the low end of an important geometric framework. (So low, in fact, there is no honest geometry to speak of, at least not in the sense of Tits.)

\begin{definition}
For a set $X$ with $|X| \ge 3$ and a collection of groups $\{U_x : x\in X\}$ with each $U_x \le \mbox{Sym}(X)$, we say that $(X,\{U_x : x\in X\})$ is a \textdef{Moufang set} if for $P := \langle U_x : x\in X \rangle$ the following conditions hold:
\begin{enumerate}
\item each $U_x$ fixes $x$ and acts regularly on $X\setminus \{x\}$,
\item $\{U_x : x\in X\}$ is a conjugacy class of subgroups in $P$.
\end{enumerate}
We call $P$ the \textdef{little projective group} of the Moufang set, and each $U_x$ is called a \textdef{root group}. The $2$-point stabilisers in $P$ are called the \textdef{Hua subgroups}.
\end{definition}

The result we require is the following, though likely we could do with less if we were willing to work a little harder.

\begin{fact}[see \protect{\cite[Theorem~A]{WiJ14}}]\label{f.Moufang}
Let $(X,\{U_x : x\in X\})$ be a Moufang set of finite Morley rank with abelian Hua subgroups and infinite root groups that contain no involutions. If the little projective group $P$ of the Moufang set has odd type, then $P\cong \psl_2(K)$ for some algebraically closed field $K$.
\end{fact}

\begin{lemma}\label{l:Moufang}
There are no rank $3$ Borel subgroups of $G$.
\end{lemma}
\begin{proof}
Assume that $B$ is a rank $3$ Borel subgroup of $G$,  
and let $X$ be the set of $G$-conjugates of $B$. We now consider the permutation group $(X,G)$, and show that it is associated to a Moufang set. For $\beta \in X$, let $U_\beta := F^\circ(\beta)$. Of course the stabiliser of $\beta$ is $N(\beta)$ and $N^\circ(\beta) = \beta$. As a consequence of Corollary~\ref{c:BorelIntersections} in type \cibo$_1$ (and the fact that $X$ has degree $1$), we find that $G$ acts $2$-transitively on $X$ with each $2$-point stabiliser of rank $1$. We now claim that $U_\beta$ acts regularly on $X\setminus \{\beta\}$. 

First, we show transitivity. As $X$ has rank $2$ and degree $1$, it suffices to show that $U_\beta$ has no orbits of rank $0$ or $1$ other than $\{\beta\}$. By connectedness of $U_\beta$, a rank $0$ orbit is a fixed point of $U_\beta$, but $U_\beta$ is too large to be contained in a $2$-point stabiliser. Next, a rank $1$ orbit gives rise to a rank $1$ subgroup $A_\beta := U_\beta \cap N(\gamma)$ for some $\gamma\neq \beta$. By $2$-transitivity,  $A_{\gamma} := U_{\gamma} \cap N(\beta)$ also has rank $1$, and by Lemma~\ref{l:uniqueness}, $A_\beta^\circ \neq A_{\gamma}^\circ$. Thus, the $2$-point stabiliser $G_{\beta,\gamma}$ has rank $2$, which is a contradiction. So $U_\beta$ acts transitively on $X\setminus\{\beta\}$.


We now show that the action of $U_\beta$ on $X\setminus\{\beta\}$ is free. Suppose $u \in U_\beta$ fixes $\gamma$ for $\gamma\neq \beta$, so $u \in U_\beta\cap N(\gamma)$. First, if $U_\beta$ is abelian, then the transitivity of $U_\beta$ on $X\setminus\{\beta\}$ forces $u$ to be in the kernel of the action of $G$ on $X$, and simplicity of $G$ implies that $u=1$. Thus, we may assume that $U_\beta$ is a nilpotent nonabelian group of rank $2$. Thus, $U_\beta$ is $p$-unipotent (see \S\ref{s:small} for the definition), and $u$ is a $p$-element. Of course, $U_{\gamma}$ is also $p$-unipotent, so by \cite[\S6.4]{BoNe94}, $u$ has an infinite centraliser in $Z^\circ(U_{\gamma})$, which has rank exactly $1$. Hence $u$ centralises $Z^\circ(U_{\gamma})$. We now have that $C^\circ(u) \geq \langle Z^\circ(U_\beta), Z^\circ(U_{\gamma})\rangle$. If $C^\circ(u)$ is soluble, then every $p$-unipotent subgroup of $C^\circ(u)$ would lie in $F^\circ(C^\circ(u))$, and this would contradict Lemma \ref{l:uniqueness}. So, $C^\circ(u)$ is a non-soluble group of rank $3$. By the work of Fr\'econ (or the fact that since $\beta\cap\gamma$ normalizes  $\langle Z^\circ(U_\beta), Z^\circ(U_{\gamma})\rangle = C^\circ(u)$, $Z^\circ(U_\beta)\cdot(\beta\cap\gamma)^\circ$ is a rank $2$ subgroup of $C^\circ(u)$), we find that $C^\circ(u) \cong \pssl_2(K)$, and as this cannot happen inside of a \cibo$_1$ group (by looking at the structure of the Sylow $2$-subgroup), we conclude that the action of $U_\beta$ on $X\setminus\{\beta\}$ is free.

This shows that $(X,\{U_\beta : \beta\in X\})$ is a Moufang set, and by connectedness results for Moufang sets (see \cite[Proposition~2.3]{WiJ11}),
each Hua subgroup is a connected rank $1$, hence abelian, group. Moreover, we claim that $U_\beta$ contains no involutions. Indeed, if $U_\beta$ contained an involution, it would contain a maximal $2$-torus $T$, the definable hull $d(T)$ of which would be central in $\beta$ by rigidity. But then, $\beta/d(T)$ would be a non-nilpotent group of rank $2$ without involutions. This is a contradiction, so the root groups of the Moufang set have no involutions. By Fact~\ref{f.Moufang}, we have a contradiction.
\end{proof}

Notice that since we are in type \cibo$_1$, using the work of Fr\'econ one could even claim that $G$ has no definable subgroups of rank $3$ at all.

\begin{proof}[Proof of Proposition~\ref{p:cibo1}]
The proof combines Brauer-Fowler computations and genericity arguments, the latter being mostly \cite[Theorem~1]{BuCh08} which asserts that the definable hull of a generic element contains a maximal decent torus (in our case a non-trivial $2$-torus will suffice). Let $i, j \in G$ be independent involutions and set $x = ij$. This strongly real element is no involution.

First we argue that $B_i = C^\circ(i)$ has rank $2$; of course, by Lemma~\ref{l:Moufang}, the rank is at most $2$.
Consider $C = C^\circ(x)$, which is normalised by $i$; by the structure of the Sylow $2$-subgroup, $C$ contains no involutions. If $C$ were not soluble it would be of the form $\psl_2(K)$ or a bad group; the first case contradicts the structure of the Sylow $2$-subgroup and the second contradicts the solubility of $B_i$ as bad groups have no involutive automorphisms \cite[Proposition~13.4]{BoNe94} (besides not existing in rank $3$). 
Thus, $C$ is soluble, so by Lemma \ref{l:Moufang} again, $\rk C \leq 2$.

Consider the map $\mu: I(G) \times I(G) \to G$ mapping $(i, j)$ to $x$. Notice that the image is the set of strongly real elements. By the structure of the Sylow $2$-subgroup and the genericity result we quoted in the first paragraph, the generic element is not strongly real, so the image set of $\mu$ has rank at most $4$. On the other hand, the fibre over $x$ is in definable bijection with the set of involutions inverting $x$, i.e. of involutions in $C^{\pm}(x)\setminus C(x)$, where $C^{\pm}(x)$ is the group of elements centralising or inverting $x$. Clearly the fibre has rank $f \leq \rk C \leq 2$. The Brauer-Fowler estimate is simply that, by additivity,
\[2 \rk I(G) - f \leq \rk\left(\mu(I(G)\times I(G)\right);\]
or put otherwise, since the map $\mu$ has non-generic image, $5 < 2 \rk B_i + f$.

In particular one must have $\rk B_i = f = \rk C = 2$. Now there is a whole coset of $C$ consisting of involutions inverting $x$, so $C$ is abelian, inverted by $i$. From this we derive a highly non-generic property. Suppose $C$ (or of course any conjugate) contains a generic element of $G$. As the latter contains in its definable hull a maximal $2$-torus, the structure of the Sylow $2$-subgroup forces $i \in C$, which is impossible. So $\bigcup_G C^g$ contains no generic element and we derive the contradiction as follows.

By the $N_\circ^\circ$-property together with Lemma \ref{l:Moufang}, $C$ must be almost self-normalising. Moreover if there is a non-trivial $c \in C\cap C^g$ for some $g \notin N(C)$ then $C^\circ(c) \geq \langle C, C^g\rangle\>$. As this centraliser is generated by two groups of rank $2$ that are each without $2$-torsion, it can be no group of rank $3$: hence $C^\circ(c) = G$. But, $G$ is simple, so $C\cap C^g = \{1\}$. Thus, $C$ is disjoint from its proper conjugates, and a standard computation now gives
\[\rk \bigcup_{g \in G} C^g = \rk\left(G/N_G(C)\right) + \rk(C) = \rk G.\]

This shows that $\bigcup_G C^g$ contains a generic element: a contradiction.
\end{proof}

\subsection{Killing \texorpdfstring{\cibo$_2$}{CiBo2}: Rank arguments}\label{s:CiBo2}

\begin{proposition}\label{p:cibo2}
The group $G$ cannot have type \cibo$_2$.
\end{proposition}
\begin{proof}
Otherwise, by Fact~\ref{f:DJ}, $5 = \rk G = 3 \rk C(i)$. 
\end{proof}

This is either a very satisfying or wholly disappointing proof, depending on your point of view; regardless, we further comment on this case below, in Section~\ref{sec.reflectionscibo2}, after introducing some geometric terminology in the next section.

\subsection{Killing \texorpdfstring{\cibo$_3$}{CiBo3}: Bachmann's theorem}

To conclude the proof of our main theorem, it remains to prove the following proposition.

\begin{proposition}\label{p:cibo3}
The group $G$ cannot have type \cibo$_3$.
\end{proposition}

\begin{setup}
Assume $G$ has type \cibo$_3$.
\end{setup}


We begin with some general observations.

\begin{remark}
Let $i\in I(G)$.
As distinct involutions give rise to distinct centralisers by Fact~\ref{f:DJ}, no maximal $2$-torus of $B_i$ is contained in $F^\circ(B_i)$. Also, $B_i = C(i)$; and since its Pr\"ufer $2$-rank is $2$, one has $\rk B_i = 3$. By Corollary~\ref{c:uniqueness}, $F^\circ(B_i)$ therefore has rank $2$. Returning to maximal $2$-tori, one sees that $B_i/F^\circ(B_i)$ and $F^\circ(B_i)$ both have Pr\"ufer $2$-rank $1$. Consequently, $Z^\circ(B_i)$ is a rank $1$ group containing divisible torsion, which in turn implies that $F^\circ(B_i)$ is abelian.
\end{remark}

We now prepare to expose some geometry. It all starts with, or is at least inspired by, the following theorem of Bachmann.

\begin{fact}[Bachmann's Theorem, see {\cite[Fact~8.15]{BoNe94}}]
If the set $I$ of involutions of a group $G$ possesses the the structure of a projective plane in such a way that three involutions are collinear if and only if their product is an involution, then $\langle I \rangle \cong \so_3(K,f)$ for some interpretable field $K$ and anisotropic form $f$ on $K^3$.
\end{fact}

As such, we make the following definition.

\begin{definition}[B-lines]
We say that three involutions $i,j,k$ of $G$ are \textdef{Bachmann-collinear}, or just \textdef{B-collinear}, if $ijk$ is an involution. In this context, involutions will also be called points.
\end{definition}

\begin{remark}
In our setting, three involutions $i,j,k$ are $B$-collinear if and only if any one of the involutions inverts the product of the other two \textit{and} at least two of the involutions do not commute. (This is not true in higher $2$-rank.)
\end{remark}

Now, it turns out that there is a related notion of collinearity that is often easier to analyse, and presumably more relevant.

\begin{definition} [C-lines]
For $i\in I(G)$, we define the \textit{polar (line) of $i$} to be $i^\perp:=I(C(i)) - \{i\}$. Such a line will be called a \textdef{centraliser line} or \textdef{C-line}. (Notice that in the case of \cibo$_1$ one would get $i^\perp = \emptyset$.) 
\end{definition}

We now show that for \cibo$_3$ in rank $5$, $I(G)$ has the structure of a projective plane with respect to $C$-collinearity. 

\begin{lemma}\label{l:linesintersect}
For any two involutions $i \neq j$, the intersection $i^\perp \cap j^\perp$ has a unique involution. That is, two distinct C-lines intersect in a unique point.
\end{lemma}
\begin{proof}
First suppose that $i$ and $j$ commute. Then by the structure of the Sylow $2$-subgroup, $I(C(i, j)) = \{i, j, ij\}$, proving $i^\perp \cap j^\perp = \{ij\}$.

Now suppose $i$ and $j$ do not commute. Let $H:= C(i,j)$. In view of the structure of $B_i$, one easily sees that $H^\circ$ has rank $1$ (see also Corollary~\ref{c:BorelIntersections}). 
If $\Pr_2(H) = 0$, then $1 < H^\circ \leq F^\circ(B_i)$ and, likewise, $H^\circ \leq F^\circ(B_j)$, against Lemma \ref{l:uniqueness}. If on the other hand $\Pr_2(H^\circ) = 2$, then $i, j \in H^\circ$ must commute. So $\Pr_2(H^\circ) = 1$, and $H^\circ$ contains a unique involution $k$. Now if $\ell \in I(H)$, then $\ell$ commutes with both $i$ and $k$, which commute themselves: hence $\ell \in \{i, k, ik\}$. The same shows that $\ell \in \{j, k, jk\}$, and therefore $\ell = k$. So $I(H) = \{k\}$, which is what we wanted.
\end{proof}

Note that if $\lambda$ is a C-line, then $\lambda$ has a unique pole (namely the involution in $C(\lambda)$), which we define to be $\lambda^\perp$. It is trivial to verify that $\perp$ is a polarity of our point-line geometry, so we quickly arrive at the following.

\begin{corollary}\label{c:projplane}
The geometry of involutions with respect to C-lines is that of a projective plane. 
\end{corollary}

At this point, we could try to prove Desargues' Theorem and get a contradiction by coordinatising Hilbert-style, then embedding $G$ into some $\mathrm{PGL}_n(K)$ and following \cite[\S8.2]{BoNe94}. However in this case, it is easier to show that C-collinearity coincides with B-collinearity. Bachmann's Theorem will then apply.

\begin{lemma}
If three involutions are B-collinear, then they are C-collinear.
\end{lemma}
\begin{proof}
Suppose that $j,k,\ell\in I(G)$ are B-collinear. We may assume that $x:=jk$ (which is inverted by $\ell$) is not an involution. Let $i\in I(G)$ be such that $\{i\} = j^\perp\cap k^\perp$. Thus, $j,k,x \in C(i) - \{i\}$, and to conclude, we show that $\ell \in C(i) - \{i\}$. Set $B_i := C(i)$.

Observe that $x \in F_i:= F^\circ(B_i)$ since $B_i/F_i$ has a unique involution. As remarked above, $F_i$ is abelian, so $C(x) \ge F_i$. As such, $C^\circ(x)$ must be soluble since it has rank at most $3$ and contains a rank $2$ abelian normal subgroup. By Lemma~\ref{l:uniqueness} and Corollary~\ref{c:uniqueness}, $B_i$ is the unique Borel subgroup containing $C^\circ(x)$, so as $C^\circ(x)$ is normalised by $\ell$, $B_i$ is as well. In \cibo$_3$, $B_i$ is self-normalizing, so $\ell \in B_i$. Furthermore, $\ell \neq i$ as otherwise $\ell$  would both invert and centralize $x$, but $x$ was chosen to not be an involution.
\end{proof}

\begin{lemma}
If three involutions are C-collinear, then they are B-collinear.
\end{lemma}
\begin{proof}
Suppose that $j,k,\ell\in i^\perp$. Set $x:=jk$ and $B_i = C(i)$. We need to show that $\ell$ inverts $x$.

Let $r$ be any involution in $i^\perp$.
As $r$ belongs to, hence centralises, some maximal decent torus $T$ of $B_i \neq B_r$ one has $C_{B_i}^\circ(r) = T \not\leq F^\circ(B_i)$. Hence in its action on $F_i:=F^\circ(B_i)$, $r$ centralises a rank $1$ subgroup, namely $(F_i\cap T)^\circ$, and $X := [B_i, r]$ is therefore the only definable, connected, infinite, proper subgroup of $F_i$ not containing a $2$-torus. Hence $X$ does not depend on $r$ and neither does $B^{-_i}:=\{b \in B_i: b^r = b^{-1}\} = X\cdot \langle i\rangle$. So $\ell$, like $j$, inverts $x$.
\end{proof}

\begin{proof}[Proof of Proposition~\ref{p:cibo3}]
Combining the previous two lemmas with Corollary~\ref{c:projplane} and Bachmann's Theorem, we are done since quadratic forms on $K^d$ are always isotropic whenever $K$ is algebraically closed (which is the case here by Macintyre's Theorem \cite[Theorem~8.1]{BoNe94}) and $d \ge 2$.
\end{proof}

\section{Reflections on \texorpdfstring{\cibo$_2$}{CiBo2}}\label{sec.reflectionscibo2}
In the case of rank $5$, \cibo$_2$ immediately succumbed to the rank computation $\rk G = 3 \rk C(i)$ from \cite{DeJa16}, but it should be mentioned, albeit briefly, that a (generically defined) projective geometry is lurking here too.

\begin{observation}
If $G$ is a group of type \cibo$_2$, then a generic pair of C-lines intersect in a unique point.
\end{observation}
\begin{proof}
By \cite{DeJa16}, if $g$ is generic in $G$, then there is exactly one $k\in I(G)$ such that  $g\in C^\circ(k)$. Additionally, since $C^\circ(k)$ is inverted by any $\omega\in I(C(k))-\{k\}$, the generic element of $G$ is a product of involutions. 

Now, let $i,j$ be a generic pair of involutions. Then $ij$ is generic in $G$, so there is exactly one involution $k$ such that  $ij\in C^\circ(k)$. This implies that $i,j\in N(C^\circ(k))$, so $i$ and $j$ normalise the maximal $2$-torus of $C^\circ(k)$. Thus, $i,j \in C(k)$, so $k \in i^\perp\cap j^\perp$. Further, if $\ell \in i^\perp\cap j^\perp$, the structure of the Sylow $2$-subgroup implies that $ij \in C^\circ(\ell)$, so $\ell = k$.
\end{proof}

We also obtain, as in \cibo$_3$, that C-collinearity coincides with B-collinearity.

\begin{observation}
If $G$ is a group of type \cibo$_2$, then C-collinearity \emph{generically}  coincides with B-collinearity.
\end{observation}
\begin{proof}
If $r, s, t\in i^\perp$, then $rs \in C^\circ(i)$ is inverted by $t$. Conversely suppose that $r$ inverts $st$ where $s, t$ are \emph{independent} involutions. Then by our previous observation, $s^\perp\cap t^\perp = \{k\}$ for some involution $k$. Now $st \in C^\circ(k)$, and as observed before, $k$ is the unique involution for which $st \in C^\circ(k)$. Thus, since $r$ inverts $st$, $r$ must fix $k$. We conclude that $r,s,t \in k^\perp$.
\end{proof}

\section{Proof of the main corollary}\label{sec:proofcorollary}

We now take up the classification of nonsoluble connected groups of Morley rank $5$. Facts~\ref{f.rankTwoGroups} and \ref{f.rankThreeGroups} as well as the following corollary to Fact~\ref{f.rankFourGroups} will be used frequently in our analysis. 

\begin{fact}[{\cite[Corollary~A]{WiJ14a} together with \cite{FrO16}}]\label{f.rankFourClassification}
If $G$ is a nonsoluble connected group of Morley rank $4$, then for $F:=F^\circ(G)$, $\rk F \le 1$, and $G$ is classified as follows.
\begin{enumerate}
\item $\rk F = 1$ and $G = F * Q$ with $Q\cong\pssl_2(K)$, or
\item $\rk F = 0$ and $G$ is a quasisimple bad group.
\end{enumerate}
\end{fact}

\begin{setup}
Let $G$ be a nonsoluble connected group of rank $5$; set $F:=F^\circ(G)$.
\end{setup}

Since connected groups of rank $2$ are soluble, we have that $\rk F \le 2$ and that $\rk G' \ge 3$.

\begin{lemma}\label{l:classifyQ3}
If $G$ has a definable, connected, normal subgroup $Q$ of rank $3$, then $G = H * Q$ with  $Q\cong\pssl_2(K)$ and $H$ connected and soluble. Consequently,  in this case we have $1 \le \rk F \le 2$ and $Q = G''$.
\end{lemma}
\begin{proof}
Since $G$ is nonsoluble, the same is true of $Q$, so (invoking \cite{FrO16}) we find that $Q\cong\pssl_2(K)$ for some algebraically closed field $K$. By \cite[II,~Corollary~2.26]{ABC08}, $G = Q* C(Q)$. As $\rk C(Q) =2$, $C^\circ(Q)$ is soluble.
\end{proof}

\begin{lemma}
If $\rk F = 0$, then $G'$ is quasisimple and bad of rank $4$ or $5$.
\end{lemma}
\begin{proof}
Assume $\rk F = 0$. By Lemma~\ref{l:classifyQ3}, $G'$ has rank $4$ or $5$, and in either case, $F^\circ(G')=1$. Now, if $Q$ is a non-trivial proper definable connected normal subgroup of $G'$, then it must be that $Q$ has rank $3$ and $Q = G''$. Consequently, Lemma~\ref{l:classifyQ3} implies that $F$ is non-trivial, a contradiction, so no such $Q$ exists. Thus, if $N$ is any proper normal subgroup of $G'$, then $[N,G']$, which is definable and connected, must be trivial, so every proper normal subgroup of $G'$ is central in $G'$. 

We now have that $G'$ is quasisimple with a finite centre, so combining our main theorem with Fact~\ref{f.rankFourClassification}, we find that $G'$ modulo its finite centre is a bad group. Hence $G'$ is bad.
\end{proof}

\begin{lemma}
If $\rk F = 1$ and $G$ is quasisimple, then $G$ is a bad group.
\end{lemma}
\begin{proof}
By Fact~\ref{f.rankFourClassification}, we find that $G/F$ is a bad group, and since $F = Z^\circ(G)$, it is easy to see that $G$ must also be bad.
\end{proof}

\begin{lemma}
If $\rk F = 1$ and $G$ is not quasisimple, then either $G = H * G''$ with $H/Z(H) \cong \agl_1(L)$ and $G''\cong\pssl_2(K)$, or $G = F * G'$ with $G'$ a rank $4$ quasisimple bad group.
\end{lemma}
\begin{proof}
Assume $\rk F = 1$. By considering the generalized Fitting subgroup of $G$, we see that $G$ contains some proper definable normal quasisimple subgroup $Q$ (see \cite[I,~Section~7]{ABC08}). 
If $Q$ has rank $3$, then, by Lemma~\ref{l:classifyQ3}, $G = H * Q$ with  $H$ connected and soluble and $Q\cong\pssl_2(K)$ for some algebraically closed field $K$. Since $\rk F = 1$, $H$ is nonilpotent, so $H/Z(H)\cong \agl_1(L)$ for some algebraically closed field $L$. Clearly, in this case, $Q = G''$. It remains to consider the case when $Q$ has rank $4$. By Fact~\ref{f.rankFourClassification}, we find that $Q$ is bad and that $F$ intersects $Q$ in a finite set. Thus, $G= F * Q$, and it must be that $Q = G'$.
\end{proof}

\begin{lemma}
If $\rk F = 2$ and $F\le Z(G)$, then $G = F * Q$ with $Q\cong\pssl_2(K)$.
\end{lemma}
\begin{proof}
Again by considering the generalized Fitting subgroup of $G$, we see that $G$ contains some definable normal quasisimple subgroup $Q$. Note that the theory of central extension prohibits $G=Q$ (see \cite[II,~Proposition~3.1]{ABC08}. If $Q$ has rank $3$, we are done, so assume that $Q$ has rank $4$. Here, $Q\cap F$ has rank $1$, so by Fact~\ref{f.rankFourClassification}, we find that $Q = F(Q) * R$ with $R$ quasisimple. But this contradicts the fact that $Q$ is quasisimple.
\end{proof}

\begin{lemma}
If $\rk F = 2$, then either $G = F * G''$ with $G''\cong\pssl_2(K)$ or $F$ is $G$-minimal.
\end{lemma}
\begin{proof}
Assume that $F$ has a definable $G$-normal subgroup $A$ of rank $1$. Then, by Fact~\ref{f.rankFourClassification}, $G/A = F/A*R/A$ for some definable connected subgroup $R$ of $G$ containing $A$ with $R/A$ quasisimple. As $R$ has rank $4$, we find that $R = A*Q$ with $Q\cong\pssl_2(K)$. Certainly, $Q = G''$.
\end{proof}

So, it remains to treat the case where $F$ is $G$-minimal (hence abelian) and non-central. The key is, of course, \cite{DeA09}.

\begin{lemma}
If $F$ is an abelian non-central rank $2$ subgroup of $G$, then there is an algebraically closed field $K$ for which $F\cong K^2$ and $G/F \cong \ssl_2(K)$ acting naturally on $F$. Moreover, if $\charac(K) \neq 2$, then the extension splits as $G = F\rtimes C(i)$ for $i$ an involution of $G$.
\end{lemma}
\begin{proof}
By  \cite[Theorem~A]{DeA09}, there is an algebraically closed field $K$ for which $F\cong K^2$ and $G/C(F)\cong \ssl_2(K)$ in its natural action. Since $C(F)$ is a finite extension of $F$, the theory of central extensions can be applied to $G/F$ to see that $C(F) = F$.  Now, assume $\charac(K) \neq 2$, and let $i$ be an involution of $G$. Set $H:=C(i)$. The image of $i$ in $G/F$ inverts $F$, so as $F$ has no involutions, we find that $G = F\rtimes C(i)$, see \cite[Lemma~9.3]{BBC07}.
\end{proof}

\section*{Acknowledgements}
This project was born in December of 2013 while the first author was a visiting professor at NYU Shanghai and the second author came for a visit---the authors are indebted to NYU Shanghai for its generous hospitality and financial support.

\bibliographystyle{alpha}
\bibliography{DWRank}
\end{document}